\documentclass{article}
\usepackage[utf8]{inputenc}

\usepackage{colonequals}
\usepackage{graphicx,verbatim}
\usepackage{amssymb}
\usepackage{amsthm}
\usepackage{bbm}
\usepackage{amsmath}
\usepackage{mathtools}
\usepackage{textgreek}
\usepackage{mathbbol}
\usepackage{amsfonts}
\usepackage{mathrsfs}
\usepackage[english]{babel}
\usepackage{lscape} 
\usepackage[vcentermath,enableskew]{youngtab}
\usepackage{tikz,tikz-cd}
\usepackage{tabu}
\usetikzlibrary{arrows,positioning,automata,shadows,fit,shapes}
\usepackage{ stmaryrd }
\usepackage[english]{babel}
\usepackage{setspace}
\usepackage{microtype}
\usepackage[hang,flushmargin]{footmisc} 

\newtheorem{theorem}{Theorem}[section]
\newtheorem{definition}[theorem]{Definition}
\newtheorem{example}[theorem]{Example}
\newtheorem{corollary}[theorem]{Corollary}
\newtheorem{lemma}[theorem]{Lemma}
\newtheorem{proposition}[theorem]{Proposition}
\newtheorem{remark}[theorem]{Remark}
\newtheorem*{theorem*}{Theorem}
\newtheorem*{definition*}{Definition}
\newtheorem*{lemma*}{Lemma}

\usepackage{chngcntr}
\counterwithin{table}{section}

\DeclareMathOperator{\triv}{triv}

\DeclareMathOperator{\Det}{det}
\DeclareMathOperator{\Dim}{dim}
\DeclareMathOperator{\Ker}{Ker}

\DeclareMathOperator{\im}{im}

\DeclareMathOperator{\sgn}{sgn}

\DeclareMathOperator{\Hom}{Hom}

\newcommand{\fg}{\mathfrak{g}}

\newcommand{\fk}{\mathfrak{k}}

\newcommand{\bbc}{\mathbb{C}}

\newcommand{\bbh}{\mathbb{H}}
\newcommand{\bbH}{\mathbb{H}}

\newcommand{\bbr}{\mathbb{R}}

\newcommand{\bbz}{\mathbb{Z}}

\newcommand{\Cc}{\mathcal{C}}
\newcommand{\Cd}{\D}

\newcommand{\Cp}{\mathcal{P}}

\newcommand{\Cw}{\mathcal{W}}

\newcommand{\Pin}{\mathsf{Pin}}

\newcommand{\D}{\mathcal{D}} 

\newcommand{\clif}{\mathcal{C}}
\newcommand{\C}{\omega} 
\newcommand{\rhoC}{\rho(\omega)}

\title{Two families of Dirac-like operators for Drinfeld's Hecke algebra}
\author{Kieran Calvert}

\begin{document}

\maketitle
\begin{abstract}
In this paper, we define two generalisations of Dirac operators for Drinfeld's Hecke algebra. One generalisation, Parthasarathy operators inherit the notion of the Dirac inequality. The second generalisation, warped Dirac operators are such that every unitary module must have a non-zero warped Dirac cohomology. An open question is whether non-zero warped Dirac cohomology can determine the infinitesimal character akin to the fact that non-zero Dirac cohomology does. For a type $A$ Hecke algebra we give a family of operators in each class.  
\end{abstract}
\section{Introduction}
Dirac operators for real Lie groups, \cite{AS77,K99,P72} have been significantly utilised in the study of $(\fg,K)$ representations. These ideas have been generalised to graded affine Hecke algebra \cite{BCT12} and Drinfeld's Hecke algebras \cite{C16,Ca20}. There are many other algebras which have an established theory of Dirac operators \cite{Ch17,CDM22, F16}.

The main results for Dirac operators are the Dirac inequality and Dirac cohomology. 
The Dirac inequality \cite{P72, Ci22} notes that the square of the Dirac operator on a unitary module is non-negative. The square is equal to a sum of two elements, both central in different algebras. In the case of Lie algebras, these algebras are $\fg$ and the diagonal embedding of $\fk$. In the case of graded affine Hecke algebras $\bbH$, the first element is central in $\bbH$ and the second central in the diagonal embedding of $\tilde{G}$. We define a generalisation of Dirac operators for Drinfeld's Hecke algebras, Parthasarathy operators (Definition \ref{d:poperator}). These operators are such that each Parthasarathy operator leads to an associated inequality (Corollary \ref{c:pineq}). For graded affine Hecke algebras associated to the symmetric group $\bbH(S_n)$ we define a family of Parthasarathy operators (Corollary \ref{c:popsym}). Unfortunately, every inequality given by this family is strictly weaker than the original Dirac inequality  \cite{BC14,Ci22} (Remark \ref{r:popsymneg}). 

The second major result for Dirac operators is Dirac cohomology and `Vogan's conjecture'. This states that if there exists non-zero Dirac cohomology \cite{Vo97} for an irreducible representation $(X, \pi_X)$ then the Dirac operator relates the infinitesimal character of $X$ with a character of a diagonal algebra. In the Lie algebra setting this diagonal algebra is $U(\fk)$ and for graded affine Hecke algebras the diagonal algebra is generated by $\tilde{G}$. Inspired by an infinite family of Dirac operators for the Dunkl angular momentum algebra \cite{CDM22} and the total Dunkl angular momentum algebra \cite{CDO22}, we define another generalisation of the Dirac operator, warped Dirac operators (Definition \ref{d:voperators}). We give an infinite family of warped Dirac operators for Drinfeld's Hecke algebras (Definition \ref{d:vopfam} and Theorem \ref{t:vadmiss}). We prove a simple condition for an irreducible unitary representation $(X,\pi_X)$ to have non-zero cohomology for at least one warped Dirac operator $\D_\C$ (Proposition \ref{p:nonzerocoh}). We show that this condition is satisfied for every irreducible unitary representation of $\bbH(S_n)$. The question arises if it is possible for the warped Dirac cohomology to inform the central character of the original module, akin to the results for the original Dirac operator.

It would be interesting to apply these ideas to Lie algebras \cite{HP06} and operators as defined by Flake \cite{F16}, Chan \cite{Ch17} and the local/global theory of Cherednik algebras introduced by Ciubotaru and De Martino \cite{CDM20}. Application of these ideas to Lie algebras could lead to a generalisation of Dirac cohomology \cite{HP06} and a larger series of representations generalising the Dirac series for real Lie groups \cite{BDW22}.

\section{Preliminaries}
\subsection{Drinfeld's Hecke algebra}

\label{drinfeldalgebras}
 In this section, we define Drinfeld algebras. Given a finite group $G$, anti-symmetric bilinear forms $b_g$ for $g \in G$ and a representation $(V,\pi_V)$ of G, then we construct an algebra
 $$ \bbH =\mathbb{C}G\rtimes T(V)/R.$$
 Here $R$ is the two-sided ideal of $\mathbb{C}G\rtimes T(V)$ generated by the relations,
 $$g^{-1}vg = \pi_V(g)(v) \text{ for all } g\in G \text{ and } v\in V,$$
 and
 $$[u,v] = \sum_{g \in G} b_g(u,v) g \text{ for all } v,u \in V.$$
 We define a filtration on the algebra $\mathbb{C}G\rtimes T(V)/R$, a vector $v$ has degree $1$ and a group element $g \in G$ has degree $0$. 
 \begin{definition}\cite{D86} \label{drinfeldalgdef} An algebra of the form $\bbH =\mathbb{C}G\rtimes T(V)/R$ is a Drinfeld algebra if it satisfies a PBW criterion. That is the associated graded algebra is naturally isomorphic to 
 $$\mathbb{C}G \rtimes S(V).$$

\end{definition}
\noindent  We state the conditions on the bilinear forms $b_g$ such that $\bbH$ is a Drinfeld algebra. 
 Define $G(b) = \{g\in G : b_g \neq 0\}$.

\begin{theorem}\cite{D86}\cite[Theorem 1.9]{RS03} The algebra $\bbH$ is a Drinfeld algebra if and only if for every $g,h \in G$ and $u,v \in V$, $h' \in Z_G(g)$, $g' \in G(b) \setminus \ker \pi_V$,

\begin{itemize} 

\item   $b_{g^{-1}hg} (u,v) = b_h (\pi_V(g)(u),\pi_V(g)(v))$

\item $\Ker b_{g'} = V^{\pi_V(g')}$ and $\Dim(V^{\pi_V(g')}) = \Dim V - 2,$

 \item $\Det(h'|_{V^{\pi_V(g')^\perp}}) = 1$, 
\end{itemize}
where $V^{\pi_V(g)^\perp }= \{ v - \pi_V(g)(v): v\in V\}$.
 
  \end{theorem}

  \begin{example}
      Let $G$ be the symmetric group $S_n$, acting on the $n-1$ dimensional root space $V$. The root space $V$ has a $S_n$ invariant form $\langle \, , \, \rangle$. Consider roots $\Phi \subset V$ for $S_n$ and let $\textbf{c}$ be a parameter in $\bbc$. Define $b_g$ for $g \in S_n$
      \[
        b_g(u,v) = \begin{cases} 0 & \text{ if } g \text{ is not a three cycle}, \\
        \textbf{c}\left(\langle\alpha , u \rangle\langle \beta , v \rangle - \langle \alpha , v \rangle\langle \beta , u \rangle\right) & \text{ for } g = s_{\alpha}s_\beta \text{ and }  \alpha \beta \in \Phi^+ \text{ such that } \langle \alpha ,\beta \rangle_{V^*} \neq 0. \end{cases}
      \]
      This defines the graded affine Hecke algebra associated to $S_n$ with representation $\bbc^{n-1}$. 
  \end{example}

\subsection{Clifford Algebra}

We define some basics on Clifford algebras, for more details see \cite{M13}.
We assume that $V$ has a $G$-invariant non-degenerate bilinear form $\langle \, , \, \rangle_V: V \times V \to \bbc.$
We consider the Clifford algebra $ \clif \colonequals \clif (V,\langle \, , \, \rangle_V)$ with canonical map $\gamma \colon V\to \clif $. Let $\{u_i\} $ be a $\langle \, , \, \rangle_V$-orthonormal basis of $V$, then $ \clif $ is generated by $\{e_i= \gamma(u_i)\}$ satisfying 
	\begin{equation*}\label{e:clifrel}
\{e_j,e_k\} = e_je_k + e_ke_j = 2 \langle e_j,e_k\rangle_V = 2 \delta_{jk}\,.
		\end{equation*}
The Clifford algebra is naturally $\bbz_2$-graded with $\gamma(V)$ having degree $\bar 1$. We extend this $\bbz_2$ grading to $\bbH\otimes \Cc$ by giving every element in $\bbH$ degree $\bar{0}$.

\begin{definition}
     For a homogeneous element $a$ in $\bbH \otimes \Cc$ with $\bbz_2$-degree $|a| \in \bbz_2$, we define $\epsilon(a) = (-1)^{|a|}$.
\end{definition}

\noindent In the Clifford algebra, there is a realisation of the group $\Pin \colonequals \Pin(V,B)$, which is a double cover of the orthogonal group $p \colon \mathrm{Pin} \to \mathsf{O}$.
We define a double cover $\widetilde{ \pi_V G} = p^{-1}(\pi_V G)$. Note that $\widetilde{\pi_V(G)}$ is not a double cover of $G$ but it is a double cover of $\pi_V(G)$. We construct a cover of $G$.  We define $\widetilde{G}$ to be the semi direct product $ \Ker \pi_V \rtimes \widetilde{\pi_V(G)} $.
Given $\widetilde{G}$ we can embed the group in $\bbH \otimes \Cc$ via 
$$\rho: \widetilde{G} \to \bbH \otimes \Cc,$$
$$\rho(\tilde{g}, h) = h p(\tilde{g})\otimes \tilde{g}, \hspace{1cm} \tilde{g} \in \widetilde{\pi_V(G)}, h \in \Ker\pi_V.$$
For a reflection $s \in G$ let $\tilde s$ denote a preimage in $\tilde G$, so $p(\tilde s) =  \pi_V s$. Let $\theta$ be the nontrivial preimage of $1$ in $\widetilde{G}$. The element $\theta$ is central in $\widetilde{G}$ and has order two: $\theta^2 =1$. 
\begin{definition}\label{d:sgn}
    We define a character $\sgn: \bbc \widetilde{G} \to \bbc$, such that $\sgn(\tilde{g}) = \det_{\pi_V}(p(\tilde{g}))$.
\end{definition}

\begin{example} A Weyl group $W$ (with simple roots $\Delta$ and positive roots $\Phi^+$) has presentations
\[
W = \langle s_\alpha,\alpha\in \Phi^+ \mid s_\alpha^2=1,s_\alpha s_\beta s_\alpha = s_\gamma,\gamma = s_\alpha(\beta)\rangle,
\]
\[
W = \langle s_\alpha,\alpha\in \Delta \mid (s_\alpha s_\beta)^{m_{\alpha,\beta}} = 1 \rangle
\]
while the double-cover has presentations \textbf{}
\begin{equation*}\label{e:PinPresentation}
\widetilde{W} = \langle \theta, \tilde{s}_\alpha,\alpha\in \Phi^+ \mid \tilde{s}_\alpha^2=1=\theta^2,\tilde{s}_\alpha \tilde{s}_\beta \tilde{s}_\alpha = \theta\tilde{s}_\gamma,\gamma = s_\alpha(\beta), \theta \text{ central}\rangle,
\end{equation*}
\begin{equation*}\label{e:PinPresentationmalpha}
\widetilde{W} = \langle \theta, \tilde{s}_\alpha,\alpha\in \Delta \mid (\tilde{s}_\alpha \tilde{s}_\beta)^{m_{\alpha,\beta}} = (\theta)^{m_{\alpha,\beta}-1}, \theta \text{ central}\rangle.
\end{equation*}
\end{example}
\noindent  The group algebra $\bbc \widetilde{G}$ splits into two subalgebras 
\[
\bbc\widetilde{G} = \frac{1}{2}(1+\theta) \bbc \widetilde{G} \oplus \frac{1}{2}(1-\theta)\bbc\widetilde{G}.
\]
The algebra $\frac{1}{2}(1+\theta) \bbc \widetilde{G} $ is isomorphic to $\bbc G$. We shall denote the algebra $\frac{1}{2}(1\pm z)\bbc\widetilde{G}$ by, respectively $\bbc\widetilde{G}_\pm$. The algebra $ \bbc G_+$ is isomorphic to $\bbc G$ and is in the kernel of $\rho$. Furthermore, $\rho$ is a homomorphism of $\bbc\widetilde{G}_-$ to $H \otimes \Cc$.

\subsection{Hermitian forms}\label{Forms}

Let $\ast$ denote the anti-automorphism $\eta^* = \varepsilon(\eta^t)$, for all $\eta\in\mathcal{C}_\bbr$. Let also $\bullet$ be the anti-linear form on $\bbH$ by $v^\bullet = -v$ for $v \in V$, and $w^\bullet = w^{-1}$, for all $w \in G$. We then define an anti-linear anti-involution $\star$ on $\bbH\otimes\Cc$ defined by taking the tensor product of these two anti-involutions.
\begin{definition}\label{d:unitary}
    A Hermitian form $\langle \, ,\, \rangle_X: X \times X \to  \bbc$ is $\bbH$ invariant if
    \[\langle h x_1,x_2\rangle_X = \langle x_1, h^\bullet x_2 \rangle_X \text{ for all } x_1,x_2 \in X \text{ and } h  \in \bbH.\]
\end{definition}
\begin{definition}
A $\bbH$-module $X$ is unitary if there exists a positive definite $\bbH$-invariant Hermitian form on $X$.
\end{definition}

\noindent  Now fix, once and for all, $(\sigma,S)$ an irreducible module (spinor module) for $\Cc$. 
\begin{definition}
    There exists a positive definite form $\langle\, ,\, \rangle_S$ on $S$. such that $\langle \gamma(v) s_1,s_2 \rangle = \langle  s_1,v^*s_2 \rangle  $ for all $v \in V$ and $s_1,s_2 \in S$. This endows $S$ with a $\ast$-unitary $\Cc$-structure.
\end{definition}

\noindent   For any $\bullet$-Hermitian module $(\pi,X)$ of $\bbH$ we endow $X\otimes S$ with a $\star$-Hermitian structure $\langle x\otimes s,x'\otimes s'\rangle_{X\otimes S} = \langle x,x'\rangle_X\langle s,s'\rangle_S$ for all $x,x'\in X$ and $s,s' \in S$.  If $X$ is $\bullet$-unitary then the $\star$-Hermitian form on $X \otimes S$ is also positive definite and hence $\star$-unitary.

\subsection{The original Dirac element}\label{s:diraczero}

If $V$ has a $G$-invariant symmetric bilinear form then one can define a Dirac operator $\mathcal{D}$. In \cite{C16} (resp. \cite{Ca20}) Dirac cohomology is defined for faithful (resp. non-faithful Drinfeld algebras). In this section, we recall definitions and a formula for $\D^2$ from \cite{Ca20}.
Given any basis $\{v_i\}$ of $V$ and dual basis $\{v^i\}$ with respect to $\langle \, ,\, \rangle_V$ we define the Dirac element
$$\mathcal{D} = \sum_{i} v_i \otimes v^i \in \bbH \otimes \Cc.$$

\noindent For every $g\in G(b)$ set, $$\textbf{k}_g = \sum_{i,j} b_g(v_i,v^j)v^iv_j \in \Cc,$$
and $$\textbf{h} = \sum_i v_iv^i \in \bbH.$$

Define the set $G(b)  = \{g\in G : b_g \neq 0\}$, write $\widetilde{G}(b)$ for the cover of this subset. For every coset representative $g \in G(b)/\Ker(\pi_V)$ define
$$\tilde{g} = \alpha\beta \in \Cc, \hspace{1cm} c_{\tilde{g}} = \frac{b_g(\alpha,\beta)}{1 - \langle \alpha ,\beta\rangle^2}\in \mathbb{C}, \hspace{1cm} e_{g} = \frac{b_g(\alpha,\beta)\langle \alpha \beta \rangle}{1 - \langle \alpha ,\beta\rangle^2}\in \mathbb{C}.$$
Every $w\in \widetilde{g}(b)$ can be written as $h^{-1} g h$ where $g$ is a coset representative of $\widetilde{G}(b)/ \Ker \pi_V$ and $h \in \Ker \pi_V$. Lemma \cite[2.3]{Ca20} gives $g = s_\alpha s_\beta$ and $\tilde{g} = \alpha\beta \in \Cc$.
We define, for $w = h^{-1}gh \in \widetilde{G}$, define
$$\tilde{g}  = \alpha \beta \in \Cc, \hspace{1cm} c_{\tilde{g}} = c_{\tilde{g}}, \hspace{1cm} e_{w} = he_{g}h^{-1}.$$
Let us define the Casimir elements, $\Omega_\bbH$ in $\bbH$ and $\Omega_{\widetilde{G}}$ in $\widetilde{G}$.
$$\Omega_{\bbH} = \textbf{h} - \sum_{g \in G(b) / \Ker\pi_V} e_g g \in \bbH^G,$$
$$\Omega_{\widetilde{G}} =   \sum_{\substack{h\in \Ker\pi_V\\ g \in G(b) /\Ker \pi_V}}  h^{-1}\tilde{g}h c_{\tilde{g}}\in \mathbb{C}[\widetilde{G}]^{\widetilde{G}}.$$
 We give a formula for $\mathcal{D}^2$. This is equivalent to \cite[Theorem 2.7]{C16},. The only variation being that $\ker \pi_V$ replaces $1$.

\begin{theorem}\label{t:Dsquare} \cite[Theorem 2.4]{Ca20}\cite[c.f. Theorem 2.7]{C16} The square of the Dirac element can be expressed as a sum of the two Casimir elements plus even terms from the Clifford algebra:
$$\mathcal{D}^2 = -\Omega_\bbH \otimes 1 +\pi_V(\Omega_{\widetilde{G}}) +  \frac{1}{2} \otimes \sum_{w \in \ker \pi_V} \textbf{k}_w.$$ 
\end{theorem}

\begin{lemma}\cite[Lemma 2.4]{C16}
The operator $\D$ $\sgn$-commutes with $\widetilde{G}$,
\[
\rho(\tilde{g}) \D \rho(\tilde{g})^{-1} = \sgn (\tilde{g})\D
\]
for every $\tilde{g} \in \widetilde{G}$.
\end{lemma}

\section{Parthasarathy Operators}

\begin{definition}\label{d:poperator}
    Let $\bbH$ be a Drinfeld algebra with group algebra $\bbc G$.  We say an operator $\Cp \in \bbH \otimes \Cc$ is a Parthasarathy operator if the following holds: 
    \begin{enumerate}
        \item $\Cp^* = \Cp$
        \item $\Cp^2 = z_1 + z_2$, where $z_1 \in Z(\bbH) \otimes 1$ and $z_2 \in Z(\rho(\bbc\widetilde{G})$. 
    \end{enumerate}
\end{definition}

\subsection{A family of Parthasarathy operators for {$\bbH$}}

Let $\bbH$ be a Drinfeld algebra, with a Dirac operator $\Cd \in \bbH \otimes \Cc$.

\begin{definition} \label{d:popfam}
For $\Xi \in \bbc\widetilde{G}_-$, we say that $\Xi = \sum \lambda_{\tilde{g}} \tilde{g}$ is $P$-admissible if 
\begin{enumerate}
    \item $\Xi^\bullet = \Xi$,
    \item $\sgn(\tilde{g}) =-1$ for all $\lambda_{\tilde{g}} \neq 0$,
    \item $\Xi^2 \in Z(\bbc\widetilde{G}_-)$.
\end{enumerate} 
\end{definition}
\begin{remark}
    The original Dirac operator (Section \ref{s:diraczero}) is a Parthasarathy Dirac operator if and only if $\kappa_g = 0$ for every $g \in \Ker \pi_V$.
\end{remark}
\noindent For the remainder of this section let us assume that $\kappa_g = 0$ for all $g \in \ker \pi_V$.
\begin{theorem} Suppose $\kappa_g =0$ for all $g \in \ker \pi_V$, for all $P$-admissible elements $\Xi \in \bbc\widetilde{G}_-$, the operator 
    \[\Cd_\Xi = \Cd + \rho\Xi,\] 
    is a Parthasarathy operator. 
\end{theorem}

\begin{proof}
Clearly, $\Cd_\Xi^* = \Cd_\Xi$. Furthermore, 
$\Cd_\Xi^2 = (\Cd + \rho\Xi)^2 = \Cd^2 + \rho\Xi^2 + \Cd\rho\Xi +\rho\Xi\Cd$. Note that $\rho\Xi \in \rho\bbc\widetilde{G}$, hence $\rho\Xi \Cd = \sum \lambda_{\tilde{g}} \tilde{g} \Cd =  \sgn(\tilde{g}) \Cd \sum \lambda_{\tilde{g}} \tilde{g}  = -\Cd \rho \Xi$. We conclude that, 
\[\Cd_\Xi^2  = \Cd^2 + \rho\Xi^2 = z_1 + z_2 \]
where $z_1 \in Z(\bbh)$ and $z_2 = \Omega_{\widetilde{G}} + \Xi ^2 \in Z\rho(\widetilde{G})$. 

\end{proof}

\subsection{$P$-admissible elements}

Throughout this section, we only consider the symmetric group $S_n$. Recall $S_n$ and $\tilde{S_n}$ have presentations: 

\[
S_n = \langle s_{ij}, {1 \leq i<j \leq n} \mid (s_{ij} s_{kl})^{m_{ijkl}} = 1 \rangle\]
where \[ 
m_{ijkl} = \begin{cases} 2 & |\{i,j,k,l\}| = 4, \\ 
3 & |\{i,j,k,l\}| = 3, \\ 
1 & |\{i,j,k,l\}| = 2.\\  \end{cases}
\]
Similary for $\widetilde{S}_n$
\[
\tilde{S_n} = \langle \theta, \tilde{s}_{ij}, {1 \leq i<j \leq n} \mid (\tilde{s}_{ij} \tilde{s}_{kl})^{m_{ijkl}} = (\theta)^{m_{ijkl}-1}, \theta \text{ central}\rangle.
\]

 \begin{definition}\cite[3.1]{BK01} The Jucys-Murphy elements in $\bbc S_n ^-$ for $i=1,\ldots ,n$ are, $$M_j = \sum_{i=1}^{j-1}  \tilde{s}_{ij}.$$ 
 \end{definition}

\begin{remark}\cite[3.1]{BK01}\label{anti-commute} The Jucys-Murphy elements anti-commute , that is 
 $$M_iM_j= -M_jM_i \text{ if } i \neq j.$$ \end{remark}
 
 \begin{lemma}\label{l:evcent} \cite[3.2]{BK01} The even centre $Z(\tilde{S_n}_-)_0$ is spanned by the set of symmetric polynomials of the Jucys-Murphy elements. \end{lemma}

\begin{proposition}
    Every odd symmetric power polynomial in the squares of the Jucys-Murphy elements is a square of an odd symmetric polynomial in the Jucys-Murphy elements. 
\end{proposition}

\begin{proof}
Suppose that $k $ is odd.
Let $P_k = \sum_{i=1}^n (M_i^2)^k$ and let $Q_k = \sum_{i=1}^n (M_i)^k$. We claim that $P_k = Q_k^2$. 

\[ Q_k^2 =  (\sum_{i=1}^n (M_i)^k)^2 =  \sum_{i=1}^n (M_i)^k  \sum_{j=1}^n (M_j)^k =  \sum_{i=1}^n (M_i)^{2k} + \sum_{i\neq j} (M_i)^k(M_j)^k+ (M_j)^k(M_i)^k\]

\noindent  Now since $k$ is odd and if $i \neq j$ then $ (M_i)^k(M_j)^k = - (M_j)^k(M_i)^k$. Hence we have shown that $Q_k^2 = P_k$.

\end{proof}

\begin{theorem}
    Let $G= S_n$, and let $Q_j$ be an odd power polynomial in the Jucys-Murphy elements $Q_j$.
    Then $\sqrt{-1}Q_j$ is $P$-admissible.
\end{theorem}

\begin{proof} 

The element, $Q_j$, consists of sums of odd polynomials in $\tilde{s}_{kl}$, all of which have odd $sgn$ group elements. Each pseudo reflection $\tilde{s}_{kl}$ is such that $\rho\tilde{s}_{kl}^\bullet = -\rho\tilde{s}_{kl}$, hence $Q_j$ is skew-adjoint and $\sqrt{-1}Q_j$ is self-adjoint. To complete the proof we note that $Q_j^2$ is central in $\bbc\widetilde{G}_-$.
\end{proof}

\begin{corollary}\label{c:popsym}
    Let $\bbH(S_n)$ be a graded affine Hecke algebra with the symmetric group. For every odd $j$, then the operator 
    \[ \D +\sqrt{-1} Q_j\]
    is a Parthasarathy operator.
\end{corollary}

\subsection{A family of Dirac inequalities}

 Suppose that $X$ is unitary (Definition \ref{d:unitary}), then any Parthasarathy operator $\Cp$ acting on $X \otimes S$ is self-adjoint. Furthermore, $X \otimes S$ has a $\bbH \otimes \Cc$-invariant positive definite form. Hence, $\pi_{X \otimes S}\Cp^2$ is a positive operator. 

\begin{corollary}[Generalised Parthasarathy inequality] \label{c:pineq}

For every Parthasarathy operator $\Cp$ and unitary module $(X,\pi_X)$, then $\Cp^2$ is positive operator on $X \otimes S$ and 
\[ \pi_X(z_1) + \pi_X(z_2) \geq 0.\]

\noindent  Here, using Definition \ref{d:poperator}, $\Cp^2 = z_1 +z_2$, with $z_1 \in Z(\bbh)$, $z_2 \in Z(\rho(\bbc\widetilde{G}))$.
    
\end{corollary}

\begin{remark}\label{r:popsymneg}
    Unfortunately, every Parthasarathy operator defined in Definition \ref{d:popfam} gives an inequality weaker than the original Dirac inequality. Let $\D_\Xi$ be defined as in Definition \ref{d:popfam}, then $\Xi$ is self-adjoint. Therefore, $\pi_X(\Xi)^2$ is a positive operator. The inequality associated to $\D_\Xi$ is
    \[ -\pi_X(\Omega_\bbH) + \pi_X(\Omega_{\widetilde{G}}) + \pi_X(\Xi)^2 \geq 0.
    \]
    Since $\pi_X(\Xi)^2$ is positive this is less restrictive than the original Dirac inequality,
      \[ -\pi_X(\Omega_\bbH) + \pi_X(\Omega_{\widetilde{G}}) \geq 0.
    \]
\end{remark}

\noindent  An interesting question, which the author intends to study, is whether there are any Parthasarathy operators in $\bbH \otimes \Cc$ which lead to new relations between the centre of $\bbH$ and the centre of $Z(\widetilde{G})$.

\section{Warped Dirac operators}

\begin{definition}\label{d:voperators}
    Let $\bbH$ be a Drinfeld algebra, we say an operator $\Cw \in \bbh \otimes \Cc$ is a warped Dirac operator if the following holds: 
    \begin{enumerate}
        \item $\Cw^* = \Cw$
        \item The operator $\Cw$ is $\sgn$-invariant under the action of $\rho\widetilde{G}$. 
    \end{enumerate}
\end{definition}

\begin{definition}

    Let $X$ be a $\bbH$-module and $S$ a spinor for $\Cc$, the operator $\Cw$ acts on $X \otimes S$ as $\Cw_{X\otimes S}$. Define $H(X,\Cw)$ as
    \[\Ker \Cw_{X \otimes S} / \Ker \Cw_{X \otimes S} \cap \im \Cw_{X \otimes S} .\]
    \end{definition}

\begin{proposition}
    The cohomology of $\Cw$ is a $\widetilde{G}$ module.     
\end{proposition}

\begin{proof}
    This follows from the fact that $\Cw$ is $\sgn-\widetilde{G}$ invariant.
\end{proof}

\subsection{A family of warped Dirac operators for $\bbH$}

\begin{definition}\label{d:vopfam}
We define the $\sgn$-centre of $\bbc\widetilde{G}_-$ to be: 

\[
Z_{\sgn}(\widetilde{G}_-) = \{ g \in \bbc\widetilde{G}_- : g h = \sgn(h) h g \quad \text { for all } h \in \bbc\widetilde{G}\}.
\]
Furthermore, we say an element is $\sgn$-central if it is contained in the $\sgn$-centre.
\end{definition}

\noindent  The ungraded centre of $Z^{ug}\bbc\widetilde{G}_-$ is equal to $\Hom_{\widetilde{G}}(\triv, \bbc\widetilde{G}_-)$ and the $\sgn$-centre of $\bbc\widetilde{G}_-$ is equal to $\Hom_{\widetilde{G}}(\sgn, \bbc\widetilde{G}_-)$.

\begin{definition}\label{d:vadmiss}
A homogeneous element $\C\in \bbc\widetilde{G}_-$ is called {\bf $wD$-admissible} if $\C$ is $\sgn$-central and $
\C^\bullet =  \C$. For any $wD$-admissible $\C\in Z_{\sgn}\widetilde{G}_-$, define
\begin{equation*}\label{e:CDirac}
\D_\C := \D + \rho \C \in \bbh \otimes \clif.
\end{equation*}

\end{definition}

\begin{theorem}
    The elements $\D_\C$ such that $\C$ is $wD$-admissible are all warped Dirac operators. Furthermore, they are precisely the modification of $\Cd$ by elements in $\rho\bbc\widetilde{G}$ which are warped Dirac operators. 
\end{theorem}

\subsection{A formula for $\D_\C^2$}

\begin{lemma} The square of $\D_\C$ is equal to a central element in $\bbH$ plus a central element in $\widetilde{G}$, a linear term in $\D_\C$ and a correction quadratic term in  $\Cc$;
\[
(\D_\C)^2 =  -\Omega_\bbH \otimes 1 +\rho(\Omega_{\widetilde{G}}+ \C^2) +   + (1+\sgn(\rhoC)) \rhoC \D +\frac{1}{2} \otimes \sum_{w \in \ker \pi_V} \textbf{k}_w.
\]
 \end{lemma}

\begin{proof}
The following calculation is simple algebra,
\begin{equation*}
\begin{aligned}
   (\D_\C)^2 &= (\D + \rhoC)^2 \\
   &= (\D)^2 + \rhoC^2 + \D \rhoC + \rhoC \D \\
            &= (\D)^2 + \rhoC^2 + \sgn(\rhoC)\rhoC\D  + \rhoC \D. \\            
\end{aligned}
\end{equation*}
Finishing the proof with the application of Theorem \ref{t:Dsquare}.
\end{proof}

\begin{lemma}[$\sgn$ invariance of $D_\C$]

 For every $\tilde{g} \in \widetilde{G}$, we have the invariance property:
\[
\rho(\tilde{g})\D_\C \rho(\tilde{g})^{-1} = \sgn (p(\tilde{g}))\D_\C.
\]
\end{lemma}

\begin{proof}
    Both $\D$ and $\C$ are $\sgn$ invariant, hence so it their sum $\D_\C = \D + \rho(\C)$
\end{proof}

\subsection{$wD$-Admissible elements}

\begin{definition}\cite[Definition 6.3]{CDO22}
Let us the define the $\theta$-centre of $\bbc\widetilde{G}$,
\[
Z^\theta(\bbc\widetilde{G}) = \{ a \in \bbc\widetilde{G} |  a \tilde{g}= \theta^{\sgn(p(\tilde{g}))}\tilde{g}a \text{ for all } \tilde{g} \in \widetilde{G}\}.
\]
\end{definition}

\begin{proposition}\cite[Proposition 6.4]{CDO22}
The $\theta$-centre of $\bbc\widetilde{G}$ is spanned by elements of the form
\[
C^\theta_{\tilde{g}} =\sum_{\tilde{h} \in \widetilde{G}}\theta^{|l(\tilde{h})|} \tilde{h}^{-1} \tilde{g} \tilde{h} =   \sum_{\tilde{h} \in \widetilde{G}_{\overline{0}}} \tilde{h}^{-1} \tilde{g} \tilde{h} + \theta  \sum_{\tilde{h} \in \widetilde{G}_{\overline{1}}} \tilde{h}^{-1} \tilde{g} \tilde{h}
\]
for any choice of $\tilde{g} \in \widetilde{G}$.
\end{proposition}

\begin{proof}
Given any $\tilde{g}$, then $C^\theta_{\tilde{g}}$ is in $Z^\theta \bbc\widetilde{G}$. Furthermore, any $a \in Z^\theta \bbc\widetilde{G}$ that has a non-zero coefficient of $\tilde{g}$, then there exists a non-zero scalar $t$ such that $a - tC^\theta_{\tilde{g}}$ is $\theta$-central with no coefficient of $\tilde{g}$. Continuing the process shows that $a$ is in the space spanned by $C^\theta_{\tilde{g}}$. 
\end{proof}

\begin{theorem}\cite[Theorem 6.5]{CDO22}

 The $\sgn$-centre of $\bbc\widetilde{G}_-$ is the projection of the $\theta$-centre of $\bbc\widetilde{G}$
\[
Z^{\sgn}(\bbc\widetilde{G}_-) = \frac{1-\theta}{2} Z^\theta \bbc\widetilde{G}.
\]
\end{theorem}

\noindent In particular, the $\sgn$-centre of $\bbc\widetilde{G}_-$ is spanned by elements of the form 
\[
C^{\sgn}_{\tilde{g}} = \sum_{\tilde{h} \in \widetilde{G}}(-1)^{|l(\tilde{h})|} \tilde{h}^{-1} \tilde{g} \tilde{h} \in \bbc\widetilde{G}_-.
\]
\noindent  Denote by $\widetilde{G}_{\bar{0}}$ the even subgroup of $\widetilde{G}$.

\begin{lemma}\cite[Lemma 6.7]{CDO22} \label{l:splitsplit}
Suppose that $\tilde{g} \in \widetilde{G}$ is even, define the $\tilde{g}$ conjugacy class, $C(\tilde{g}) = \{ \tilde{w }\in \widetilde{G}: \tilde{g} = \tilde{h}^{-1}\tilde{g}\tilde{h}, \tilde{h} \in \widetilde{G}\}$, then the element $Z^\epsilon_{\tilde{g}}$ is non zero if and only if the conjugacy class $C(\tilde{g})$ splits into two conjugacy classes in $\widetilde{G}_{\overline{0}}$.

\end{lemma}

 \begin{theorem}\cite{Sc11} \cite[Theorem 2.7]{St89}
 Let $\lambda$ be an even partition of $n$. The $\tilde{S}_n$ conjugacy classes $C_\lambda$ (or $C_\lambda^{\pm})$ if already split) split into two $\tilde{A}_n$ conjugacy classes if and only if $\lambda \in DP_n^+$. Here $DP_n^+$ is the set of distinct partitions of $n$ which are even. 
 
 \end{theorem}

\begin{corollary} \label{c:sgncent}
The $sgn$-centre of $\bbc\widetilde{G}_-$ has basis \[\{C^{\sgn}_{\tilde{g}} : C(\tilde{g}) \text{ splits into two conjugacy classes in   } \widetilde{G}_{\overline{0}} \}.\]
\end{corollary}
\noindent Due to Corollary \ref{c:sgncent} we may assume that any $wD$-admissible element in $\bbc\widetilde{G}_-$ is even. 
 
\begin{theorem} \label{t:vadmiss}
The $wD$-admissible elements in $\bbc\widetilde{G}_-$ are equal to the real-span of the set \[\{C^{\sgn}_{\tilde{g}} + C^{\sgn}_{\tilde{g}^{-1}}, \sqrt{-1}(C^{\sgn}_{\tilde{g}} - C^{\sgn}_{\tilde{g}^{-1}}) : C(\tilde{g}) \text{ splits into two conjugacy classes in   } \widetilde{G}_{\overline{0}} \}.\]
\end{theorem}

\begin{proof}
    The $wD$-admissible elements are elements in the $\sgn$-centre which are self adjoint. Since $\rho \tilde{g}^\bullet$ = $\rho \tilde{g}^{-1}$, then taking a basis for the $\sgn$-centre from Corollary \ref{c:sgncent} and adding or subtracting $C^{\sgn}_{\tilde{g}^{-1}})$ enforces this set to be self adjoint (or skew adjoint respectively). Multiplying by $\sqrt(-1)$ forces the skew adjoint operators to be self adjoint. This then spans all self-adjoint operators in the $\sgn$-centre. 
\end{proof}

 \begin{corollary}
\label{l:typeAex}
Let $G=S_n$ and let $\{g\}$ be the set of elements in $S_n$ associated to an even partition $\lambda$ which has distinct cycles. Then $C(\tilde{g})$ splits in $\tilde{A}_n = (\tilde{S}_n)_{\overline{0}}$. Then from Lemma \ref{l:splitsplit}, $C^{\sgn}_{\tilde{g}} \neq 0$. The group $\widetilde{S}_n$ is ambivalent, that is, $\tilde{g}^{-1}$ is always conjugate to $\tilde{g}$. Therefore every nonzero element $C^{\sgn}_{\tilde{g}} = \frac{1}{2}(C^{\sgn}_{\tilde{g}} + C^{\sgn}_{\tilde{g}^{-1}})$ is an admissible element and $C^{\sgn}_{\tilde{g}} - C^{\sgn}_{\tilde{g}^{-1}} =0$ for every $C^{\sgn}_{\tilde{g}}$.
\end{corollary}

\section{Vogan's Dirac morphism}

In this section, we show that the original proof of Vogan's morphism does not apply to $\D_\C$ unless $\C =0$. Note that this is a correction on the previous version of this preprint. Akin to the original proof given in \cite{HP02} we define an odd derivation related to the (warped) Dirac operator $\D_\C$.

 \subsection{The linear map $d_\C$}

\begin{definition}
Let $\D_\C$ be a warped Dirac operator as defined in Definition \ref{d:vopfam}. We define a map from $\bbH \otimes \Cc$ to $\bbH \otimes \Cc$.

\[
d_\C: \bbh \otimes \clif \to \bbh \otimes \clif,
\]
where $d_\C(a) = \D_\C a - \epsilon(a) \D_\C$, for $a \in \bbh \otimes \clif$.

\end{definition}

\begin{remark}
The map $d_\C$ is an odd derivation, i.e.,
\[
d_\C(ab) = d_\C(a)b + \epsilon(a)d_\C(b)
\]
for all $a,b \in \bbh \otimes \clif$.
\end{remark}
\begin{lemma} The image under $\rho$ of $\bbc \widetilde{G}$ is in the kernel of $d_\C$
\[
\rho (\bbc \widetilde{G}) \subset \ker d_\C.
\]
\end{lemma}

\begin{proof}
This follows from the fact that $\D_\C$ is $\sgn-\widetilde{G}$ invariant. 
\end{proof}
\noindent  The operator $\D_\C$ intertwines the $\sgn$ and $\triv$ $\widetilde{G}$ isotypic components $\bbH \otimes \Cc ^{\triv}$ and $\bbH \otimes \Cc^{\sgn}$.
We define $d_\C^{\triv}$ and $d_\C^{\sgn}$ to be $d_\C$ restricted to the $\triv$ and $\sgn$ isotypic components respectively. 
Since the kernel of $d_\C$, contains $\rho(\bbc\widetilde{G})$, then the kernel of $d_\C^{\triv}$, contains the $\triv$-isotypic component $\rho(\bbc\widetilde{G}^{\widetilde{G}})$.

\begin{theorem}\label{t:dker}
    The only warped Dirac operator such that the kernel of $d_\C^{\triv}$ equals:
\[
\ker d_\C^{\triv} = \im d_\C^{\sgn} \oplus \rho(\bbc\widetilde{G}^{\widetilde{G}})
\]
Is the original Dirac operator $\D_0$.

 \end{theorem}

 \begin{proof}
     This follows from the fact that, for $\C \neq 0$, $\D_\C^2$ is not in the centre of the algebra $\bbH \otimes \Cc$. Therefore, the derivation $d_\omega$ does not square to zero, thus the kernel cannot be contained in the image and so 
     \[ \ker d_\omega^{\triv} \neq \im d_\omega^{\sgn} \oplus \rho(\bbc\widetilde{G}^{\widetilde{G}}\] 
     for $\C \neq 0$.
 \end{proof}

\subsection{Dirac cohomology} 
Let $(X,\pi_X)$ be a representation of $\bbH$, We say that $X$ is an admissible module if the decomposition of $X$ into $\Omega_\bbh$-generalized eigenspaces
\[ X = \bigoplus_{\lambda \in \bbc} X_\lambda\]
is such that every $X_\lambda$ is finite-dimensional. 
\begin{definition}
    For an irreducible $\bbH$-representation $(X,\pi_X)$, let $\chi: Z(\bbH) \to \bbc$ be the infinitesimal character $\pi_X |_{Z(\bbH)}$.
\end{definition}
\begin{definition}
Let $X$ be an admissible $\bbH$-module and let $S$ be a spinor for $\Cc$, then $X \otimes S$ is a $\bbH \otimes \Cc$ module and $\D_\C \in \bbh \otimes \Cc$ acts
    \[
(D_\C)_X : X \otimes S \to X \otimes S, 
\] The Dirac $\C$-cohomology of $X$ (and $S$) is defined as
\[
H(X,\C) = \ker (D_\C)_X/ \ker (D_\C)_X \cap \im (D_\C)_X.
\]

\end{definition}
\noindent Since $\D_\C$ $\sgn$-commutes with $\widetilde{G}$ the Dirac $\C$-cohomology of an admissible $X$ is a finite dimension $\widetilde{G}$ module, or zero. Outside of $\C=0$ is it possible that non-zero warped Dirac cohomology $H(X,\C)$ relates the infinitesimal character of $X$ with a character of $\tilde{G}$? If this were true, the facts below that, if there are enough $wD$-admissible elements then every unitary module has non-zero warped Dirac cohomology for some $wD$-admissible $\C$, would become a very strong tool for studying unitary modules.

\noindent Recall Theorem \ref{t:vadmiss} states that all $wD$-admissible elements for $\bbc\tilde{G}_-$ are even. Suppose that $\C$ is $wD$-admissible, because every group element occurring in $\C$ is even then $\C$ and $\D$ commute.

\begin{proposition}\label{p:unitarycoh}
If $X$ is a $\bullet$-unitary $\bbH$-module, then $H(X,\C) = \ker(\D_\C)$.
\end{proposition}

\begin{proof}  If $X$ is $\bullet$-Hermitian then the image and kernel of $\D_\C$ are orthogonal with respect to $\langle \,,\,\rangle_{X \otimes S}$ and hence $\ker(\pi\D_\C)\cap \im(\pi\D_\C)={0}$
\end{proof}

\begin{proposition}\label{p:nonzerocoh}
Let $(\pi_X,X)$ be a $\bullet$-unitary module for $\bbH$. Suppose that there exists an admissible element $\C \in \bbc \widetilde{G}_-$ such that $\pi_X(\C) \neq 0$. Then, there exists a $wD$-admissible $\lambda\C \in \bbc \widetilde{G}_-$ such that $H(X,\lambda\C) \neq 0$.
\end{proposition}

\begin{proof}
Since $X$ is Hermitian then $H(X,\C) = \ker(\D_C) = \ker (\D_C^2)$. We study the kernel of the operator $(\D_\C -2\rho(\C))\D_\C =\D^2 - \rho(\C)^2$. The elements $\D^2$ and $\rho(\C)$ commute, hence have simultaneous eigenvalues. Given $\pi(\rhoC)^2 \neq 0$ it has positive real eigenvalues. Let $\pi_{X \otimes S}\D^2$ and $\pi_{X \otimes S}(C)^2$ have positive simultaneous eigenvalues $d \geq 0$ and $c > 0$ respectively. One can modify $\D_\C$ to $\D_{\lambda\C}$, with $\lambda = \sqrt{\frac{c}{d}}$. This ensures that 
\[
\D_0^2 - \lambda^2\rho(\C)^2
\]
has a non-zero kernel on $X \otimes S$. Thus proving that there exists a non-zero kernel for the operator $(\D_\C -2\rho(\C))\D_\C$. Since the composition of injective functions is injective, we find that one of  $(\D_\C -2\rho(\C))$ or $\D_\C$ has a non-zero kernel.   
\end{proof}

\begin{remark}\label{r:actionofC}
When $G=S_n$, we have shown that the set of $wD$-admissible elements in $\bbr\tilde{W}_-$ is equal to conjugacy classes associated to even partitions. 
From \cite{K05}, $Z_0(\bbr\tilde{W}_-)$ is equal to $\{\sigma(M_1^2,\ldots,M_n^2)\mid \sigma \textup{ real symmetric polynomial}\}$. The element $\sigma(M_1^2,\ldots,M_n^2)$ acts on an irreducible representation by evaluating $\sigma$ at specific real values \cite[Corollary 6.3]{C18}. In particular, we can conclude that for every $\bbc\tilde{W}$-module there is a $wD$-admissible element that does not act by zero.
\end{remark}

\begin{theorem}
    Let $\bbH$ be a degenerate affine Hecke algebra associated to the symmetric group $S_n$. Then for any unitary $\bbH$-module $X$, there exists a warped Dirac operator $\D_\C$ such that 
    \[ H(X,\C) \neq 0.\]
\end{theorem}

\subsection{warped Dirac inequalities}
Throughout this section, we assume that $\kappa_{g} = 0$ for all $g \in \ker \pi_V$.

\noindent Recall Theorem \ref{t:vadmiss} states that all $wD$-admissible elements for $\bbc\tilde{G}_-$ are even. Let $\C = \sum c_{\tilde{g}} \tilde{g}$, this then implies $\sgn(\tilde{g}) = 1$ for every $c_{\tilde{g}} \neq 0$.
Suppose that $\C$ is $wD$-admissible, because every group element occurring in $\C$ is even then $\C$ and $\D$ commute. On a $\star$-unitary module $X \otimes S$, the operator, $\D_\C^2$ is positive.

\begin{proposition}
    Let $(X, \pi_X)$ be a $\bullet$-unitary module, then $(X \otimes S, \pi_{X \otimes S})$ is $\star$-unitary and
    $
    \D_\C^2 = \D^2 +2 \D \rho(\C) + \rho(\C)^2 \geq 0. $ 
  Hence for every $wD$-admissible $\C$ the following inequality holds ,
    \[ \pi_X(\Omega_{\bbH}) \leq \pi_{X \otimes S}\Omega_{\tilde{G}} +   \pi_{X \otimes S}((\C + 2\D)\C) .\]
       
\end{proposition}
\noindent Specialising the above inequality to $\C =0$ gives the Dirac inequality. Proposition \ref{p:nonzerocoh} states that if there exists an $\C$ such that $\pi_{X \otimes S}(\C) \neq 0$ then there exists a $wD$-admissible element such that the above inequality becomes an equality. For the Hecke algebra associated to the symmetric group there is always non-zero $\C$-cohomology, this implies that for a particular choice of $\C$, it is always possible to make the inequality above strict.

\section{Parthasarathy and warped Dirac operators}

\begin{proposition}

    Consider elements of the form $\Cd + \rho A_{\widetilde{G}}$, where $A_{\widetilde{G}} \in \bbc\widetilde{G}_-$. The only operator of this form that is both a Parthasarathy and warped Dirac operator is 
\end{proposition}

\begin{proof}
Suppose that  $\Cd + \rho A_{\widetilde{G}}$ is a Parthasarathy operator,  by construction $\sgn$ of every group element in $A_{\tilde{g}}$ is $-1$. Now suppose that $\Cd + \rho A_{\widetilde{G}}$ is a warped Dirac operator, then by Theorem \ref{t:vadmiss}, $\sgn$ of every group element is $1$. Therefore $A_{\widetilde{G}}$ has no nonzero coefficient of group elements and hence must be zero.

\end{proof}

\section*{Acknowledgements}
This research was supported by the Heilbronn Institute for Mathematical Research. 
\bibliographystyle{abbrv}

\bibliography{references}

\small 

\noindent Kieran Calvert, \texttt{kieran.calvert@lancaster.ac.uk}\\
Department of Mathematics and Statistics, 
Lancaster University,
Bailrigg, Lancaster, UK\\

\end{document}